\documentclass[a4paper,reqno]{amsart}

\usepackage{amsfonts}
\usepackage{mathrsfs}
\usepackage{graphicx}
\usepackage[normalem]{ulem}
\usepackage{url}
\usepackage{amsthm,amsfonts,amssymb,amsmath,esint,amscd,latexsym,amsxtra}
\usepackage[colorlinks=true, urlcolor=blue,bookmarks=true,bookmarksopen=true,
citecolor=blue]{hyperref}
\usepackage{amscd}
\usepackage{enumerate,bbm}

\usepackage{fullpage}
\usepackage[all]{xy}
\usepackage[OT2,T1]{fontenc}
\usepackage{xspace}

\newcommand{\BC}{{\mathbb {C}}}

 \newcommand{\BN}{{\mathbb {N}}}
 
 \newcommand{\BR}{{\mathbb {R}}}

\newcommand{\CC}{{\mathcal {C}}}

\newcommand{\CK}{{\mathcal {K}}}

\newcommand{\RU}{{\mathrm {U}}}

\newcommand{\BBC}{{\mathrm{BC}}}

 \newcommand{\EP}{{\mathrm{EP}}}

 \newcommand{\GL}{{\mathrm{GL}}}
\newcommand{\GSp}{{\mathrm{GSp}}}
\newcommand{\Hom}{{\mathrm{Hom}}}

\newcommand{\PGL}{{\mathrm{PGL}}}  
\newcommand{\pr}{{\mathrm{pr}}}

\newcommand{\Rep}{{\mathrm{Rep}}}

\newcommand{\Res}{{\mathrm{Res}}}

\newcommand{\SL}{{\mathrm{SL}}}
 
\newcommand{\SO}{{\mathrm{SO}}}
\newcommand{\St}{{\mathrm{St}}}

\newcommand{\RTr}{{\mathrm{Tr}}}

\newcommand{\Ext}{{\mathrm{Ext}}}

\newcommand{\wh}[1]{{\widehat {#1}}}

\newcommand{\sk}{\medskip}
\newcommand{\lra}{\longrightarrow}

\newcommand{\bs}{\backslash}

\newcommand{\s}{\sk\noindent}

\newcommand{\res}{\mathrm{res}}

\newcommand{\dfn}[1]{\textit{#1}}

\makeatletter\def\varW@#1#2{%
\vtop{\m@th\ialign{##\cr
\hfil$#1 \mathrm{colim} $\hfil\cr
\noalign{\nointerlineskip\kern1.5\ex@}#2\cr
\noalign{\nointerlineskip\kern-\ex@}\cr}}
}
\makeatother
\makeatletter
\def\colim{%
\mathop{\mathpalette\varW@{}}\nmlimits@
}\makeatother

\theoremstyle{plain}
\newtheorem{thm}{Theorem}[section] \newtheorem{cor}[thm]{Corollary}

\newtheorem{lem}[thm]{Lemma}  \newtheorem{prop}[thm]{Proposition}

\theoremstyle{remark} \newtheorem{remark}[thm]{Remark}
\theoremstyle{definition} 
\theoremstyle{definition} \newtheorem{example}[thm]{Example} 
\newtheorem{defn}[thm]{Definition}\newtheorem{defn+lem}[thm]{Definition and Lemma}

\numberwithin{equation}{section}

\newcommand*{\sheafhom}{\mathrm{H}\kern -.5pt om}

\begin{document}
\title{Top degree Ext-groups and relatively supercuspidal spectra}
\date{}

\author{Li Cai}
\address{Academy for Multidisciplinary Studies\\
Beijing National Center for Applied Mathematics\\
Capital Normal University\\
Beijing, 10048, People's Republic of China}
\email{caili@cnu.edu.cn}

\author{Yangyu Fan} 
\address{Key Laboratory of Algebraic Lie Theory and Analysis of Ministry of Education,\\ School of Mathematics and Statistics, Beijing Institute of Technology,\\ Beijing, 100081, People's Republic of China
} 
\email{yangyu.fan@bit.edu.cn}

\maketitle

\begin{abstract}
Relatively supercuspidal representations are analogue of  supercuspidal representations in the relative Langlands program. 
This work studies relatively supercuspidal representations using top degree Ext-groups via the Schneider-Stuhler
duality. As examples, the relatively supercuspidal spectra  for the Flicker-Rallis case and the diagonal case are determined. 
\end{abstract}

\tableofcontents

\section{Introduction}\label{intro}
Let $F$ be a $p$-adic field. Let $G$ be a reductive group over $F$ and $\Rep(G(F))$ be the category of complex smooth $G(F)$-representations. 
It is an important theorem of
Berstein-Zelevinsky \cite{BZ} that for any irreducible $\pi\in\Rep(G(F))$, there exists a parabolic subgroup $P=MN\subset G$ with Levi factor $M$ and  $\sigma\in\Rep(M(F))$ supercuspidal such that $\pi\hookrightarrow I_{P(F)}^{G(F)}\sigma$ (normalized induction). Moreover, the pair $(M,\sigma)$ is unique up to $G(F)$-conjugacy, called
the \dfn{cuspidal support} of $\pi$. 

It is natural to ask how to classify irreducible smooth representations in terms of  cuspidal supports. 
At least for discrete series representations, the  complete theory is available for both
$\GL_n$ (see \cite{Zel}) and classical groups (see \cite{MT} and \cite{X}). 

In the {\em relative Langlands program}, one considers  irreducible  $\pi\in\Rep(G(F))$ distinguished
by a subgroup $H\subset G$, i.e., $\Hom_{H(F)}(\pi,\BC) \not= 0$. The objects analogous to supercuspidal representations are the 
{\em relatively supercuspidal} (RSC) representations, i.e., distinguished irreducible  $\pi\in\Rep(G(F))$
such that  for any  $\ell\neq0 \in \Hom_H(\pi,\BC)$, the relative matrix
coefficients $\ell(\pi(\cdot)v)$, $v \in \pi$ are compactly supported modulo $Z_GH(F)$, where $Z_G$ is the
center of $G$.  Consider the case $(G,H)$ is a symmetric pair, say $H$ is the fixed point of
an involution $\theta$ on $G$. Kato-Takano  \cite[Theorem 7.1]{KT08} 
proves an analogue of the above theorem of 
Bernstein-Zelevinsky: for any distinguished  $G(F)$-representation $\pi$ on $G(F)$,  there exist a parabolic subgroup $P$ of $G$ 
with $M = P \cap \theta(P)$ a Levi subgroup of $P$ and a $H \cap M$-RSC $\sigma\in\Rep(M(F))$, 
such that  $\pi \hookrightarrow I_{P(F)}^{G(F)}\sigma$. 
This result gives a way to reduce the local harmonic analysis on
$p$-adic symmetric spaces to RSC representations (See, e.g. \cite{Off11}).  

It is a fundamental problem to determine RSC representations. 

 
 For symmetric pairs,  Kato-Takano  \cite[Theorem 6.9]{KT08} gives a characterization of relative supercuspidality in terms of Jacquet modules. In \cite{M, Smith, KT20}, many RSC representions are constructed by considering parabolic inductions of specified supercuspidal representations.  One key  tool in all these works is computing the Jacquet modules.

In this paper, we will detect RSC representations via top-degree Ext-groups instead of Jacquet modules. It works even for the pairs $(G,H)$ which
are not symmetric (See the diagonal case below). 

Our approach is motivated by the study of Prasad \cite{Pra18} on Ext-analogues of
branching law  and relates directly to the problem to determine irreducible (non-supercuspidal) 
$H(F)$-subrepresentations  of irreducible $\pi\in\Rep(G(F))$ (See \cite[Proposition 9.1]{NP20} for 
an inspiring example and also the work of Speh-Venkataraman \cite{SV1,SV2} when $F$ is Archimedean).

There are two natural questions for RSC representations.

\s{\bf Q1:} Classify RSC representations in terms of  cuspidal supports. 

\vspace{12 pt}

The second question is to understand RSC representations in the general framework of Sakellaridis-Venkatesh 
on the Plancherel formula for spherical varieties \cite{SV}. 
Assume the geometric quotient $X = H \bs G$ is spherical with respect to the right $G$-multiplication.
Under certain conditions, there exists  a reductive group $G_X$ over $F$ and a functorial morphism  
\[\iota_*: \wh{G_X} \lra \wh{G}\]
Here,   $\wh{G_X}$ and $\wh{G}$
are the unitary duals of $G_X$ and $G$ respectively.

Roughly speaking, it is  conjectured that there is an isomorphism of unitary representations
\[L^2(X(F)) \stackrel{\sim}{\lra} \int_{\wh{G_X}} \iota_*(\sigma)^{\oplus m(\sigma)}  d \sigma\]
where $d\sigma$ is the Plancherel measure of $G_X$ and $m(\sigma)$ is certain multiplicity.


\s{\bf Q2}: Classify RSC representations in term of  $G_X(F)$-representations  via $\iota_*$. 

\vspace{12 pt}

In fact, we will see from the cases below that
\begin{itemize}
	\item (Q1) implies (Q2) assuming
a good understanding of the morphism $\iota_*$. 
	\item a good understanding of the conjectured Plancherel decomposition (at least for the discrete
		spectrum) is a key ingredient for (Q1). 
\end{itemize}

The starting point of our method is the following key observation based on the 
Schneider-Stuhler duality theorem given by Nori-Prasad (See
\cite{NP20} and also Theorem \ref{SSNP} below).
\begin{lem}[See Proposition \ref{tprsc} for the general case] \label{lem1.1}
   Assume  $Z_G \cap H = Z_H$. For any irreducible $\pi \in \Rep(Z_H(F)\bs G(F))$ with $\dim\Hom_{H(F)}(\pi,\BC)<\infty$,
   \[\dim \Hom_{H(F)}(\pi,\BC) \geq \dim \Ext_{Z_H(F)\bs H(F)}^{d^\prime(\pi)}(D(\pi)^\vee,\BC).\]
   Moreover, $\pi$ is RSC if and only if
   \[0\neq\dim \Hom_{H(F)}(\pi,\BC) = \dim \Ext_{Z_H(F)\bs H(F)}^{d^\prime(\pi)}(D(\pi)^\vee,\BC).\]
   Here, 
   \begin{itemize}
   \item $\Ext_{Z_H(F)\bs H(F)}^{d^\prime(\pi)}(D(\pi)^\vee,\BC):= \Ext_{\Rep(Z_H(F)\bs H(F))}^{d^\prime(\pi)}(D(\pi)^\vee,\BC)$ is the degree $d^\prime(\pi)$ Ext-group in $\Rep(Z_H(F)\bs H(F))$.
   \item If denoted by $\omega$ the central character of $\pi$ and $\Rep(G,\omega)$ the category of 
	   smooth $G(F)$-representations with central character $\omega$,
    then $d^\prime(\pi)$ is the top Ext-degree  of $\pi$  in $\Rep(G,\omega)$, i.e. 
    for any $\pi' \in \Rep(G(F),\omega)$ and any $i > d^\prime(\pi)$,
   $\Ext^i_{\Rep(G(F),\omega)}(\pi,\pi') = 0$. In fact, $d^\prime(\pi)$ is the split rank of $Z_M\cap[G,G]$ where 
    $M$ is the Levi subgroup carrying the cuspidal support of $\pi$.
       \item $D(\pi)$ is the Aubert-Zelevinsky involution of $\pi$ with
    $D(\pi)^\vee$ its smooth dual. 
   \end{itemize}
\end{lem}
\begin{proof}
Extend  $\omega$ to $Z_GH(F)$ by  composing with the trivial character on $H(F)$. Denote by $I_{Z_GH(F)}^{G(F)}\omega$ 
(resp. $i_{Z_GH(F)}^{G(F)} \omega$) the space
of smooth functions $f:\ G(F)\to\BC$ (resp. with compact support modulo $ZH(F)$) such that 
\[ f(zg) = \omega(z)f(g) \quad \forall z\in Z_GH(F),\ g\in G(F)\]
on which  $G(F)$ acts by right translation.  Then
\[\Hom_{H(F)}(\pi,\BC) = \Hom_{Z_GH(F)}(\pi,\omega)=\Hom_{G(F)}\left( \pi, I_{Z_GH(F)}^{G(F)}\omega \right) \supset \Hom_{G(F)}\left( \pi, i_{Z_GH(F)}^{G(F)}\omega \right).\]
And $\pi$ is $H$-RSC if and only if   
\[0\neq\Hom_{G(F)}\left( \pi, I_{Z_GH(F)}^{G(F)}\omega \right)  = \Hom_{G(F)}
\left( \pi, i_{Z_GH(F)}^{G(F)}\omega \right).\]

By the  Schneider-Stuhler duality for $\pi$,
\[\dim\Hom_{G(F)}\left( \pi, i_{Z_GH(F)}^{G(F)}\omega  \right) = \dim
\Ext_{\Rep(G(F),\omega)}^{d^\prime(\pi)}\left( i_{Z_GH(F)}^{G(F)}\omega, D(\pi)\right)\]
By duality and the Frobenius reciprocity law,
$$\Ext_{\Rep(G(F),\omega)}^{d^\prime(\pi)}\left( i_{Z_GH(F)}^{G(F)}\omega, D(\pi)\right)=\Ext_{\Rep(G(F),\omega^{-1})}^{d^\prime(\pi)}
\left( D(\pi)^\vee, I_{Z_GH(F)}^{G(F)} \omega^{-1}\right)=\Ext_{Z_H(F)\bs H(F)}^{d^\prime(\pi)}\left( D(\pi)^\vee, \BC\right)$$
\end{proof}

We shall study (Q1) and (Q2) using Lemma \ref{lem1.1} for the following two cases
\begin{itemize}
   \item the Flicker-Rallis case: $E/F$ is a quadratic field extension and 
	    \[(G,H) = \left(\Res_{E/F}\GL_n, \GL_n\right).\]
	 See Theorem \ref{topd} and \ref{topd-B}.
	    \item the diagonal case:  $G_2$ is a subgroup of $G_1$ and
	   \[(G,H) = \left( G_1 \times G_2, G_2 \right)\]
	   where  $G_2$ embedded into $G_1 \times G_2$ diagonally. See
	   Theorem \ref{diagonal} and \ref{GGP}.
\end{itemize} 

The Flicker-Rallis pair is symmetric and Gelfand, i.e. $\dim\Hom_{H(F)}(\pi,\BC)\leq 1$ 
for any irreducible $\pi\in\Rep(G(F))$.  In general, the diagonal case is not symmetric so that
the characterization of RSC representations in terms of Jacquet modules given by Kato-Takano
cannot be applied directly.

\subsection{The Flicker-Rallis case} 

Let $E/F$ be a quadratic field extension. Let $G = \Res_{E/F} \GL_n$ and $H = \GL_n$.


The following is our main result for the Flicker-Rallis case, which answers (Q1). 
\begin{thm}\label{topd}  Let $\pi\in \Rep(G(F))$  be an
 irreducible  representation. Then $\pi$   is RSC if and only if 
$\pi = I_{P(F)}^{G(F)}\sigma$ 
for some parabolic subgroup $P=MN \subset G$ and some  
$H \cap M$-distinguished regular supercuspidal $M(F)$-representation $\sigma$.  
\end{thm}
Here  an irreducible representation
$\sigma = \boxtimes_i \sigma_i\in\Rep(M(F))$
is called {\em regular } if $\sigma_i$ are mutually distinct.

Immediate from Theorem \ref{topd}, we have the following
result determining  irreducible $\pi\in\Rep(G(F))$ admitting the Steinberg representation $\St = D(\BC) \in \Rep(H(F))$ as a subrepresentation. Let $T\subset \Res_{E/F}\GL_n$ be the diagonal torus.
\begin{cor}\label{Stein-sub} Denote by $\St$ the Steinberg representation of $H(F)$.  For any irreducible $\pi\in G(F)$, $\dim\Hom_{H(F)} (\St,\pi) \leq 1$ with the equality holds iff $\pi = I_{B(F)}^{G(F)} \chi$ with $\chi$ regular and $T\cap H$-distinguished. 
\end{cor}
\begin{proof} 
The top degree $d^\prime(\St)=d^\prime(\BC)=n-1$. By the Schneider-Stuhler duality for $\BC\in\Rep(H(F))$, for  any irreducible  $\pi\in\Rep(G(F))$
$$\dim \Hom_{H(F)}\left( \St, \pi \right) = \dim \Ext_{F^\times\bs H(F)}^{n-1}\left( \pi, \BC\right).$$
(Readers may compare with the use of Schneider-Stuhler duality in the proof of Lemma \ref{lem1.1}.)

Note that
$$\Ext^{n-1}_{F^\times\bs H(F)}(\pi,\BC)= \Ext^{n-1}_{F^\times\bs G(F)}(\pi,I_{H(F)}^{G(F)}\BC)=\Ext^{n-1}_{F^\times\bs G(F)}(i_{H(F)}^{G(F)}\BC, \pi^\vee)$$
and the top degree $d^\prime(\pi) \leq n-1$. If $d^\prime(\pi) < n-1$, $\Hom_{H(F)}\left( \St, \pi \right)$. When $d^\prime(\pi)=n-1$,  by Theorem \ref{topd} and applying the
Schneider-Stuhler duality for $\pi^\vee$, 
$$\dim\Hom_{H(F)}\left(\St, \pi \right)=\dim\Hom_{G(F)}(D(\pi)^\vee,i_{H(F)}^{G(F)}\BC)\leq1$$
with the equality holds if and only if $\pi = I_{B(F)}^{G(F)} \chi$ with $\chi$ regular and $T \cap H$-distinguished.
\end{proof}
\begin{remark} 
\begin{itemize}
\item It seems also possible to prove Corollary \ref{Stein-sub} by computing Jacquet modules. In case $n=2$, we supply all details in Proposition \ref{prasad}.

\item By considering the higher Ext-groups for the triple product pair, when $n=2$, 
we strengthen  Corollary \ref{Stein-sub} by replacing $\St$ to arbitrary irreducible representations in \cite{CF22}. 
\end{itemize}
\end{remark}

The proof of Theorem \ref{topd} contains three parts:
\begin{enumerate}
	\item The ``if'' part: $\pi = I_{P(F)}^{G(F)}\sigma$ with $\sigma\in \Rep(M(F))$ regular supercuspidal 
		and $H\cap M(F)$-distinguished is RSC.
	\item If $\pi = I_{P(F)}^{G(F)}\sigma$ is  RSC with $\sigma$ square-integrable, then $\sigma$ must be 
	distinguished supercuspidal.  In particular, RSC square-integrable representations are exactly 
		distinguished supercuspidal representations.
	\item   An irreducible $\pi\in\Rep(G(F),\omega)$ is a relatively discrete series (RDS) representation, 
		i.e. $\pi$ is a subrepresentation of   the $L^2$-completion $L^2(Z_GH(F) \bs G(F),\omega)$ of $i_{Z_GH(F)}^{G(F)}\omega$ if and only if	
		$\pi = I_{P(F)}^{G(F)}\sigma$ for  $H \cap M$-distinguished 
		regular square-integrable  $\sigma\in\Rep(M(F))$. 
\end{enumerate}
For the ``only if'' part of Theorem \ref{topd},  note that if $\pi$ is RSC, then  $\pi$ is RDS.
By (3), $\pi = I_{P(F)}^{G(F)}\sigma$ with $\sigma$ distinguished regular and square-integrable. 
By (2),  $\sigma$ must be supercuspidal.

Part (3) follows from the work of Mok \cite{Mok} on the local base change lifting from
unitary groups to general linear groups and the work of Beuzart-Plessis \cite{BP18P} on the
conjectured Plancherel formula for $L^2(H(F) \bs G(F))$. 
We briefly  review  these two results  here using notations from \cite[Section 2.10]{BP18P}.

Denote by $\RU_n$ the quasi-split unitary group of rank $n$ with respect to
$E/F$. Denote by $\mathrm{Temp}(\RU_n)$
 the set of irreducible admissible tempered representations
on $\RU_n(F)$. The local Langlands correspondence for $\RU_n$  established by Mok \cite{Mok}  gives
a finite-to-one map from $\mathrm{Temp}(\RU_n)$ to the set of
tempered $L$-parameters for $\RU_n$ satisfying certain properties, whose fibres are the so-called
tempered $L$-packets
of $\RU_n(F)$. Denoted by $\mathrm{Temp}(\RU_n)/\mathrm{stab}$ the set of tempered $L$-packets.  
An $L$-packet of $\RU_n(F)$ is called
square-integrable (resp. supercuspidal) if all its members are square-integrable
(resp. supercuspidal). In fact, an $L$-packet is square-integrable if and only if
one of its members is square-integrable. 

Denote by 
\[\BBC_n: {^L}\RU_n \lra {^L}G\]
the so-called stable (resp. unstable) base change morphism if $n$ is odd (resp. even). This
morphism induces the base change map from $L$-parameters of $\RU_n$ to that of $G$.
By the local Langlands correspondences for $\RU_n$ and $G$, $\BBC_n$ induces the 
base change lifting 
\[\BBC_n: \mathrm{Temp}(\RU_n)/\mathrm{stab} \lra \mathrm{Temp}(G).\]

It is 
conjectured by Flicker-Rallis (See \cite[Section 3]{A}) and proved by Matringe \cite[Theorem 5.2]{Mat11} that for  $\pi\in\Rep(G(F))$ generic, $\pi$ is $H$-distinguished if and only
if its $L$-paramater is in the image of  $\BBC_n$. 
In particular, by the above
work of Mok, for any $\pi \in \mathrm{Temp}(G)$ (so that $\pi$ is generic), $\pi$ is 
$H$-distinguished if and only if $\pi$ is in the image of $\mathrm{BC}_n$. 


There is an isomorphism of unitary representations \cite[Theorem 1]{BP18P}
\[L^2(H(F) \bs G(F)) \cong \int_{\mathrm{Temp(\RU_n)/\mathrm{stab}}}^\oplus \BBC_n(\sigma) d\sigma.\]
In particular, a $G(F)$-representation $\pi$ is  RDS if and only 
if $\pi$ is in the image of square-integrable $L$-packets under $\mathrm{BC}_n$. 

By \cite{Mok}, 
$\pi \in \mathrm{Temp}(G)$ lies in the image  of  square-integrable 
	$L$-packets  under  $\mathrm{BC}_n$ if and only 
	if $\pi=I_{P(F)}^{G(F)}\sigma$  for some parabolic subgroup $P=MN \subset G$ and some  
$H \cap M$-distinguished regular square-integrable $M(F)$-representation $\sigma$. 
Consequently, a $G(F)$-representation $\pi$ is  RDS if and only 
	if $\pi=I_{P(F)}^{G(F)}\sigma$  for some parabolic subgroup $P=MN \subset G$ and some  
$H \cap M$-distinguished regular square-integrable $M(F)$-representation $\sigma$. 

By studying Jacquet modules, (1) is obtained by Kato-Takano \cite[ 
Section 2.4]{KT20} and Smith \cite[Corollay 6.7]{Smith}. 
The new point here is that we use homological methods to prove both (1) and (2).  (We will not say anything about
(3).)

More precisely, we compute the  top degree Ext-group
$\Ext^{d^\prime(\pi)}_{F^\times\bs H(F)}\left( \pi,\BC\right )$ for $\pi = I_{P(F)}^{G(F)}\sigma\in\Rep(F^\times\bs G(F))$ (not necessarily irreducible) 
where $P = MN\subset G$ is any parabolic subgroup and  $\sigma\in\Rep(M(F))$ irreducible. 
By Mackey theory, the restriction $\pi|_H$ is ``glued'' by
induced representations associated to the double coset $P \bs G / H$. In particular, 
we have an exact short sequence
\[\tag{E} 1 \lra \tau \lra \pi|_H \lra I_{P_H(F)}^{H(F)} \delta_{P_H}^{1/2} \sigma|_{P_H} \lra 1\]
where $P_H = P \cap H$ and $\delta_{P_H}$ is the modulus character of $P_H$. Consider  the 
following piece of the induced long exact sequence
\[ \cdots \to\Ext_{F^\times\bs H(F)}^{d^\prime(\pi)-1}(\tau,\BC) \lra 
\Ext_{F^\times\bs H(F)}^{d^\prime(\pi)} ( I_{P_H(F)}^{H(F)}\delta_{P_H}^{1/2} \sigma|_{P_H} ,\BC) 
\lra \Ext_{F^\times\bs H(F)}^{d^\prime(\pi)}(\pi,\BC) \lra \Ext_{F^\times\bs H(F)}^{d^\prime(\pi)}(\tau,\BC)\to \cdots.\]
By the Kunneth formula (see Proposition \ref{Kun}) and the Schneider-Stuhler duality theorem, one can deduce from the geometry of double cosets that
\begin{enumerate}[(A)]
	\item $\dim \Ext_{F^\times\bs H(F)}^{d^\prime(\pi)}(I_{P_H(F)}^{H(F)}\delta_H^{1/2}\sigma|_{P_H} ,\BC) \leq 1$ with equality holds if and only if 
	$D(\sigma)^\vee$ is RSC (see Lemma \ref{com1} for details)
	\item when $\sigma$ is supercuspidal, $\Ext_{F^\times\bs H(F)}^{d^\prime(\pi)}(\tau,\BC) = 0$ and if moreover $\sigma$ is distinguished and 
	regular,  then $\Ext_{F^\times\bs H(F)}^{d^\prime(\pi)-1}(\tau,\BC) = 0$ (see Lemma \ref{com2} for details).
\end{enumerate}

Note that  when $\sigma$ is distinguished supercuspidal, $\pi = I_{P(F)}^{G(F)} \sigma$ 
is irreducible by the Bernstein-Zelevinsky classification (see Theorem \ref{BZ}) and $D(\pi)=\pi$  (Proposition \ref{AuZe2}), one immediately deduce  (1) (See Proposition \ref{reRSC}). 

Conversely, when $\sigma$ is supercuspidal,  $\Ext^{d'(\pi)}_{F^\times \bs H(F)}(\pi,\BC) \not= 0$  for $\pi = I_{P(F)}^{G(F)} \sigma$ implies 
 $\sigma$ is distinguished. Together with Bernstein-Zelevinsky classification, one deduces  (2) (See Proposition \ref{descent}).

\begin{remark}
In the proof (Section \ref{sec:FR}), to apply the Kunneth formula, it is convenient to consider  $\Ext^{d(\pi)}_{H(F)}(\pi,\BC)$, where $d(\pi)$ is the top degree in $\Rep(G(F))$ for $\pi$, which equals to $\Ext^{d^\prime(\pi)}_{F^\times\bs H(F)}(\pi,\BC)$ (See Lemma \ref{td} for details).
\end{remark}



Now, we consider (Q2). The relevant functorial map $\iota_*$ is $\BBC_n$.

It is reasonable to expect  the image of supercuspidal $L$-packets under $\mathrm{BC}_n$ 
admits a similar description as the square-integrable $L$-packets. Actually for symplectic groups and
orthogonal groups, explicit description for supercuspidal representations in square-integral $L$-packets are available 
(see \cite[Theorem 1.5.1]{Moeg} or \cite[Theorem 3.3]{X}), which is stronger than the following: 

\s{\bf (H).} {\em 	
	Let $\pi \in \mathrm{Temp}(G)$. Then $\pi$ is in the image  of  supercuspidal 
	$L$-packets  
	under  $\mathrm{BC}_n$ if and only if $\pi = I_{P(F)}^{G(F)}\sigma$ 
	for some parabolic subgroup $P=MN \subset G$ and some  
$H \cap M$-distinguished regular supercuspidal $M(F)$-representation $\sigma$.  }

\begin{thm}\label{topd-B} 
	Assuming {\bf (H)},  $\pi$ is  RSC if and only if it is in the
image of supercuspidal $L$-packets under $\mathrm{BC}_n$.
\end{thm}

Clearly, the Theorem \ref{topd-B} (which answers (Q2)) follows immediately from
Theorem \ref{topd} (which answers (Q1)). 

There seems no reference for the proof of {\bf (H)}. It seems approachable, for example, via the method of 
\cite{X} with results in \cite{Mok}.

\subsection{The diagonal case}

Let $(G,H) = (G_1 \times G_2,G_2)$ with $G_2\subset G_1$ embedded into $G_1\times G_2$ diagonally. 
\begin{thm}[See Theorem \ref{diagonal2}]
\label{diagonal} 
Assume $Z_G\cap H(F)\bs Z_H(F)$ is compact. 
Then  RSC representations are
      supercuspidal. If moreover 
 \begin{itemize}
     \item either $(G,H)$ is symmetric;
     \item or the geometric quotient $X:=H\bs G$ is a wavefront spherical variety and $G$ is split,
 \end{itemize} 
 RSC representations are exactly  distinguished  supercuspidal representations.
 \end{thm}

Applying Lemma \ref{lem1.1}, it is quite straightforward to show  RSC representations are supercuspidal. 
For simplicity, 
assume $Z_G \cap H = Z_H$ and $\pi = \pi_1 \boxtimes \pi_2\in\Rep(G(F))$ is  RSC with trivial central character. 
Then
\begin{align*}
    0 \not= \Ext_{Z_H(F) \bs H(F)}^{d'(\pi)}(D(\pi)^\vee,\BC) &= 
	\Ext_{Z_H(F) \bs H(F)}^{d'(\pi)}(D(\pi_1)^\vee,D(\pi_2))\\
	&= \Ext_{ Z_{G_1}(F) \bs G_1(F)}^{d'(\pi)}
\left(D(\pi_1)^\vee, I_{Z_{G_2}(F) \bs G_2(F)}^{Z_{G_1}(F) \bs G_1(F)} D(\pi_2) \right)
\end{align*}
The nonvanishing implies $d'(\pi) \leq d'(\pi_1)$ and $d'(\pi) \leq d'(\pi_2)$. As $d'(\pi) = d'(\pi_1)
+d'(\pi_2)$, we have $d'(\pi_1) = d'(\pi_2) = 0$, or equivalently, $\pi$ is supercuspidal. 

The "moreover" part follows immediately from  \cite[lemma 3.1]{Zha0}. More precisely, the first case is a direct consequence of
 the work of Kato-Takano, and  the second case is a consequence of the work of Sakellaridis-Venkatesh.

The assumptions in Theorem \ref{diagonal} are satisfied in the
following cases: 
\begin{itemize}
\item the group case: $G_1 = G_2$;
\item the orthogonal Gan-Gross-Prasad case: $G_1=\SO_{n+1}$ and $G_2=\SO_n$;
\item the triple product case: $G_1=\GL_2\times\GL_2$ and $G_2 \subset G_1$ is the diagonally embedded $\GL_2$.
\end{itemize}
Moreover, for all the above cases, RSC representations are distinguished representations arising from supercuspidal representations via $\iota_*$.

When $Z_G\cap H(F)\bs Z_H(F)$ is non-compact, 
one can only deduce $\pi_2$ is supercuspidal and  $d^\prime(\pi_1)\leq d(\pi_2)$  for RSC $\pi=\pi_1\boxtimes\pi_2\in\Rep(G(F))$ (See Proposition \ref{RSCsc}). 

In the case $G_1=\GL_{n+1}$ and $G_2=\GL_n$, one can proceed using the result of Prasad \cite[Theorem 4.1]{Pra18} on the Euler-Poincare characteristic number. 

\begin{thm}[Theorem \ref{GGP2}] \label{GGP}
Assume $G_1=\GL_{n+1}$ and $G_2=\GL_n$.  Let  $\pi=\pi_1\boxtimes\pi_2\in\Rep(G(F))$ be an irreducible  $H$-distinguished representation. Then $\pi$ is  RSC if and only if\begin{itemize}
     \item $\pi$ is supercuspidal when $n\geq2$,
     \item $\pi$ is  supercuspidal or $\pi_1=\St\otimes\eta^{-1}$ and $\pi_2=\eta$ for some character $\eta$ on $F^\times$ when $n=1$.
 \end{itemize}
 \end{thm}
This result suggests the relation between RSC representations and the functorial map $\iota_*$ could be complicated:
not all the RSC representations are from supercuspidal $L$-packets via $\iota_*$ 
(which is the identity map here). It would be interesting to explore such phenomenon.

\subsection{Notations} \label{notation}
Throughout this paper, let $F$ be a $p$-adic field and  $E/F$ be a quadratic field extension with the non-trivial $F$-automorphism $-$. 

For any  reductive group $G$ over $F$, denote the center  by $Z_G$.  For  any subgroup $Z\subset Z_G$ and any character $\omega:\ Z(F)\to\BC^\times$, let $\Rep(G(F),\omega)$ be the category of $G(F)$-representations on which $Z(F)$ acts by $\omega$.  By \cite[Section 4]{Ber92}, $\Rep(G(F),\omega)$ is abelian and contains enough projective objects and  injective objects.   For any $\pi,\sigma\in  \Rep(G(F),\omega)$, let  $\Ext_{\Rep(G(F),\omega)}^i(\pi,\sigma)$ denote the $i$-th extension group of $\pi,\sigma$ in the category  $\Rep(G(F),\omega)$.
When $Z=\{e\}$ and $\omega$ trivial, denote $\Rep(G(F),\omega)$ (resp. $\Ext_{\Rep(G(F),\omega)}^i(\pi,\sigma)$) by $\Rep(G(F))$ (resp. $\Ext_{G(F)}^i(\pi,\sigma)$). In the rest of the paper, the notation $\Rep(G(F),\omega)$  means $\omega$ is a character of $Z_G(F)$ unless otherwise specified.

For $\pi\in \Rep(G(F))$, let $\pi^\vee$ be the smooth dual. For any $\pi\in\Rep(G(E))$, let $\bar{\pi}$ be the twist of $\pi$ by the involution $-$.

For any irreducible $\pi\in\Rep(G(F))$ with cuspidal support $(M,\sigma)$, let $d(\pi)$ (resp. $d^\prime(\pi)$) be the split rank of $Z_M$ (resp. $Z_M\cap[G,G]$).

For any subgroup $H\subset G$, let $I_{H(F)}^{G(F)}$ (resp. $i_{H(F)}^{G(F)}$) denote the normalized induction (resp. compact induction) functor from $\Rep(H(F))$ to $\Rep(G(F))$.

For the reductive  group $\GL_n$ over $F$, let $T\subset B\subset \GL_n$  denote the maximal split torus of diagonal matrices and the Borel subgroup of upper triangular matrices respectively. Parabolic (resp. Levi) subgroups $P\subset \GL_n$ (resp. $M\subset P$) contains $B$ (resp. $T$) are called {\em standard}. Throughout, we will abbreviate the condition $P\subset \GL_n$ is a standard parabolic subgroup with standard Levi subgroup $M$ and unipotent radical $N$ as  $P=MN\subset \GL_n$ is a standard parabolic subgroup.




\s{\bf Acknowledgement} We thank Professor Ye Tian for his consistent encouragement. We want to especially  thank Professor Dipendra Prasad for all the inspiring exchanges. In particular, we want to mention the proof of Proposition \ref{prasad} is due to him. We also thank Professor Chong Zhang and Professor Kei-Yuen Chan for helpful discussions. Finally, we thank the anonymous referee  for very helpful and valuable comments. L. Cai is partially supported by NSFC grant No.11971254.




\section{Preliminaries}
Let  $G$ be a reductive group over  $F$. In the following, we will record some general facts on Ext-groups for the categories $\Rep(G(F))$ and $\Rep(G(F),\omega)$ with $\omega:\ Z_G(F)\to\BC^\times$. 
\begin{lem}\label{restriction}
 Take any projective object $\pi\in \Rep(G(F),\omega)$. Then for any  closed subgroup  $H\subset G$,
\begin{enumerate}[(i)]
    \item    $\pi|_{H}\in\Rep(H(F),\omega|_{Z_G\cap H})$ is projective and for any  $\sigma\in\Rep(H(F), \omega|_{Z_G\cap H})$,
$$\Ext^i_{\Rep(G(F),\omega)}(\pi,\mathrm{un}-I_{Z_GH}^G\sigma)\cong \Ext^i_{\Rep(H(F),\omega|_{Z_G\cap H})}(\pi|_H,\sigma)$$
where $\mathrm{un}-I_{Z_GH}^G$ is the unnormalized induction functor.
\item $\pi|_H\in\Rep(H(F))$ is  projective if  $Z_G\cap H(F)$ is compact. 
\end{enumerate}
\end{lem}
\begin{proof}Note that restriction induces an  equivalence of categories
$$ \Rep(H(F),\omega|_{Z_G\cap H})\cong \Rep(Z_GH(F),\omega).$$
Then by Frobenius reciprocity law, Item (i) follows from \cite[Proposition 2.2]{Pra18}.

For Item $(ii)$, note that for any $\sigma\in\Rep(H(F))$,
 $$\sigma_\omega=\{v\in \sigma\mid z\cdot v=\omega(z)v,\quad \forall\ z\in Z_G\cap H(F)\}\subset \sigma$$ is a direct summand and $\Hom_{H(F)}(\pi|_H,\sigma)=\Hom_{H(F)}(\pi,\sigma_\omega)$. Thus to show $\pi|_H\in \Rep(H(F))$ is projective, it suffices to show that for any surjection 
$\sigma_2\twoheadrightarrow \sigma_1$
in $\Rep(H(F), \omega|_{Z_G\cap H})$, one has 
$$\Hom_{H(F)}(\pi|_H,\sigma_2)\twoheadrightarrow \Hom_{H(F)}(\pi|_H,\sigma_1).$$
By the category equivalence
$\Rep(H(F), \omega|_{Z_G\cap H})\cong \Rep(Z_GH(F),\omega)$
and the Frobenius reciprocity law, there is a canonical isomorphism
$$\Hom_{\Rep(H(F),\omega_{Z_G\cap H})}(\pi|_H,\sigma_i)\cong \Hom_{\Rep(G(F),\omega)}(\pi,\mathrm{un}-I_{Z_GH}^G\sigma_i).$$
As $\pi\in\Rep(G(F),\omega)$ is projective, we are done.
\end{proof}
The following lemma is well-known.
\begin{lem}\label{boxtensor}Let $G$, $G^\prime$ be  reductive groups over $F$. If both $\pi\in \Rep(G(F),\omega)$ and $\pi^\prime\in\Rep(G^\prime(F),\omega^\prime)$ are  projective, $\pi\boxtimes\pi^\prime\in\Rep(G\times G^\prime(F),\omega\boxtimes\omega^\prime)$ is projective.
\end{lem}
\begin{proof}For any  $\sigma\in \Rep(G\times G^\prime(F))$, one has a canonical isomorphism
$$\Hom_{G\times G^\prime(F)}(\pi\boxtimes\pi^\prime,\sigma)=\Hom_{G(F)}(\pi,\Hom_{G^\prime(F)}(\pi^\prime,\sigma)).$$
 Then for any surjection 
$\sigma_2\twoheadrightarrow \sigma_1$
in $\Rep(G\times G^\prime(F), \omega\boxtimes\omega^\prime)$, $$\Hom_{\Rep(G\times G^\prime(F),\omega\boxtimes\omega^\prime)}(\pi\boxtimes\pi^\prime,\sigma_2)\twoheadrightarrow \Hom_{\Rep(G\times G^\prime(F),\omega\boxtimes\omega^\prime)}(\pi\boxtimes\pi^\prime,\sigma_1)$$
 by the assumption on $\pi$ and $\pi^\prime$.  Consequently, $\pi\boxtimes\pi^\prime\in\Rep(G\times G^\prime(F),\omega\boxtimes\omega^\prime)$ is projective.
\end{proof}
Combining Lemma \ref{restriction} and \ref{boxtensor}, one has the following straightforward corollary.
\begin{cor}\label{projective}If both $\pi\in \Rep(G(F),\omega)$ and $\sigma\in\Rep(G(F),\omega^{-1})$ are projective, then $\pi\otimes\sigma\in\Rep(Z_G(F)\bs G(F))$ is projective.
\end{cor}
\begin{prop}\label{Ext}For any $\pi\in\Rep(G(F),\omega)$ and  $\sigma\in\Rep(G(F),\omega^{-1})$, there are canonical isomorphisms
$$\Ext^i_{\Rep(Z_G(F)\bs G(F))}(\pi\otimes\sigma,\BC)\cong \Ext^i_{\Rep(G,\omega)}(\pi,\sigma^\vee)\cong \Ext^i_{\Rep(G,\omega^{-1})}(\sigma,\pi^\vee),\quad \forall\ i\in\BN.$$
\end{prop}
\begin{proof} Take any projective resolution $Q_\bullet$ of $\sigma$ in $\Rep(G(F),\omega^{-1})$. Assume $\pi\in\Rep(G(F),\omega)$ is projective. Then  $\pi\otimes Q_\bullet$ is a projective resolution of $\pi\otimes\sigma\in\Rep(Z_G(F)\bs G(F))$ by Corollary \ref{projective}. Hence the complex $\Hom_{\Rep(Z_G(F)\bs G(F))}(\pi\otimes Q_{\bullet},\BC)$
computes $\Ext^i_{\Rep(Z_G(F)\bs G(F))}(\pi\otimes\sigma,\BC)$.
Note that 
$$\Hom_{\Rep(Z_G(F)\bs G(F))}(\pi\otimes Q_\bullet,\BC)\cong \Hom_{\Rep(G,\omega^{-1})}(Q_\bullet,\pi^\vee)\cong \Hom_{\Rep(G,\omega)}(\pi,Q_\bullet^\vee),$$
so it also computes $\Ext^i_{\Rep(G,\omega^{-1})}(\sigma,\pi^\vee)$ and $\Ext^i_{\Rep(G,\omega)}(\pi,\sigma^\vee)$. 

Note that $\pi$ is projective implies
$$\Ext^i_{\Rep(Z_G(F)\bs G(F))}(\pi\otimes\sigma,\BC)\cong \Ext^i_{\Rep(G,\omega)}(\pi,\sigma^\vee)=0,\ \forall\ i\geq1.$$
Consequently, $\pi\otimes\sigma$ is an acyclic object for the functor $\Hom_{Z_G(F)\bs G(F)}(-,\BC)$. 

For arbitrary $\pi\in\Rep(G(F),\omega)$, take any projective resolution $P_\bullet$. Then $P_\bullet\otimes\sigma$ is an acyclic resolution of $\pi\otimes\sigma$ with respect to $\Hom_{Z_G(F)\bs G(F)}(-,\BC)$. Therefore $$\Hom_{Z_G(F)\bs G(F)}(P_\bullet\otimes\sigma,\BC)\cong  \Hom_{\Rep(G,\omega^{-1})}(\sigma,P_\bullet^\vee)\cong \Hom_{\Rep(G,\omega)}(P_\bullet,\sigma^\vee)$$
 computes $\Ext^i_{\Rep(Z_G(F)\bs G(F))}(\pi\otimes\sigma,\BC)$ as well as $\Ext^i_{\Rep(G(F),\omega^{-1})}(\sigma,\pi^\vee)$ and $\Ext^i_{\Rep(G(F),\omega)}(\pi,\sigma^\vee)$. 
 \end{proof}
 The  following Kunneth formula is quite useful.
\begin{prop}\cite[Proposition 3.1]{Pra18}\label{Kun}Let $G$, $G^\prime$ be  reductive groups over $F$. Let $\pi_1,\pi_2\in \Rep(G(F))$ and $\sigma_1,\sigma_2\in\Rep(G^\prime(F))$ such that  one of $\pi_1$ and $\sigma_1$ has  finite length or both $\pi_2$ and $\sigma_2$ have finite length. Then
$$\Ext^i_{G\times G^\prime(F)}(\pi_1\boxtimes\sigma_1,\pi_2\boxtimes\sigma_2)=\bigoplus_{j+k=i}\Ext^j_{G(F)}(\pi_1,\pi_2)\otimes\Ext^k_{G^\prime(F))}(\sigma_1,\sigma_2).$$
\end{prop}
\begin{proof}
When $\pi_1$ or $\sigma_1$ has finite length, this is shown in  \cite[Proposition 3.1]{Pra18}. 

If both $\pi_2$ and $\sigma_2$ have finite length, then $\pi_2=(\pi_2^\vee)^\vee$ and $\sigma_2=(\sigma_2^\vee)^\vee$. Hence by duality,
$$\Ext^i_{G\times G^\prime(F)}(\pi_1\boxtimes\sigma_1,\pi_2\boxtimes\sigma_2)=\Ext^i_{G\times G^\prime(F)}(\pi_2^\vee\boxtimes\sigma_2^\vee, \pi_1^\vee\boxtimes\sigma_1^\vee).$$
\end{proof}

 \section{Aubert-Zelevinsky involution and Schneider-Stuhler duality}
 In this section, we introduce the Aubert-Zelevinsky involution and Schneider-Stuhler duality following \cite{NP20}. Let $G$ be a connected reductive group over $F$ and fix a minimal parabolic subgroup $P_\emptyset\subset G$.  Then for any finite length $\pi\in \Rep(G(F))$, the \dfn{Aubert-Zelevinsky involution $D(\pi)$} is the virtual representation in the Grothendieck group  $$I_{P_\emptyset}^{G(F)}(R_{P_\emptyset}(\pi))-\sum_{P_1}I_{P_1}^{G(F)}(R_{P_1}(\pi))+\sum_{P_2}I_{P_1}^{G(F)}(R_{P_2}(\pi))\cdots$$
 where  $P_1$ is the next larger parabolics  in $G$ containing $P_\emptyset$, $P_2$ is next larger parabolics containing $P_1$ etc and $R_{P_*}$ is the normalized Jacquet functor with respect to $P_*$. For  $\pi\in\Rep(G(F))$ irreducible, $D(\pi)$ is actually given by an honest irreducible $G(F)$-representation (up to sign), which one still denotes by $D(\pi)$, and  it is known that   $D(D(\pi))=\pi$ and $D(\pi^\vee)=D(\pi)^\vee$. Moreover, for any parabolic subgroup $P=MN\subset G$ with opposite $P^-$ and any finite length $\sigma\in\Rep(M(F))$, $D(I_{P(F)}^{G(F)}\sigma)\cong I_{P^-(F)}^{G(F)}D(\sigma)$. 
 
 In many cases, one can describe the Aubert-Zelevinsky involution explicitly.
 \begin{defn}Let $P=MN\subset G$ be any parabolic subgroup with Levi factor $M$. Let $W_G$ (resp. $W_M$) be the Weyl group of $G$ (resp. $M$) for a maximal split torus in $G$. An irreducible representation $\sigma\in \Rep(M(F))$ is called \dfn{regular} if $\sigma^w\ncong\sigma$ for any nontrivial $w\in W_M\bs W_G/W_M$. 
 \end{defn}
 \begin{prop}\cite[Proposition 2.1]{NP20}\label{AuZe1}For any regular supercuspidal irreducible $M(F)$-representation $\sigma$, $I_{P(F)}^{G(F)}\sigma$ has a unique irreducible sub-representation $S(\sigma)$ and a unique irreducible quotient representation $Q(\sigma)$. Moreover, one has $D(S(\sigma))=Q(\sigma)$ and $D(Q(\sigma))=S(\sigma)$.
 \end{prop}

 In the case $G=\GL_n$, a standard Levi group $M$ has the form $\prod_{i=1}^d\GL_{n_i}$ with $\sum_i n_i=n$. Moreover, an irreducible $M(F)$-representation $\sigma=\boxtimes\sigma_i$ is regular if and only if  the $\GL_{n_i}(F)$-representations $\sigma_i$ are mutually distinct. When all $n_i$ are equal and $\sigma_i=\sigma_1|\det|^{i-1}$,  $\sigma$  is called a \dfn{segment} for $\GL_n$. Note that the Steinberg representation $\St$ can be realized as the unique irreducible quotient of  
 \[I_{B(F)}^{\GL_n(F)}|\cdot|^{\frac{1-n}{2}}\boxtimes|\cdot|^{\frac{3-n}{2}}\cdots\boxtimes|\cdot|^{\frac{n-1}{2}}.\]
 \begin{thm}\cite{BZ,Zel} \label{BZ}
 Let  $P=MN\subset \GL_n$ be any standard parabolic subgroup and  $\sigma=\boxtimes_{i=1}^j\sigma_i\in\Rep(M(F))$ be any irreducible supercuspidal representation. Then $I_{P(F)}^{\GL_n(F)}\sigma$ is reducible if and only if for some $i\neq j$, $\sigma_i\cong\sigma_j|\det|^{\pm1}$. When $I_{P(F)}^{\GL_n(F)}\sigma$ is irreducible, $I_{P(F)}^{\GL_n(F)}\sigma\cong I_{Q(F)}^{\GL_n(F)}\sigma$ for any parabolic
 subgroup $Q$ with Levi factor $M$.
 
 If $\sigma$ is a segment, then $S(\sigma)$ is not essentially square-integrable while $Q(\sigma)$ is essentially square-integrable, which is square-integrable if and only if $\sigma_1(\frac{d-1}{2})$ is unitary.
 Moreover, all square-integrable $\GL_n(F)$-representations arise this way. 
  \end{thm}
 
 \begin{prop}\label{AuZe2}Let  $P=MN\subset \GL_n$ be any standard parabolic subgroup and  $\sigma=\boxtimes_{i=1}^j\sigma_i\in\Rep(M(F))$ be any irreducible supercuspidal representation. If $\pi=I_{P(F)}^{\GL_n(F)}\sigma$ is irreducible, $D(\pi)=\pi$. 
 \end{prop}
  \begin{proof} This follows immediately from Theorem \ref{BZ} by noting that $D(\pi)=I_{P^-(F)}^{\GL_n(F)}\sigma$ is irreducible. 
 \end{proof}
 The Aubert-Zelevinsky involution plays a crucial role in the Stuhler-Schneider duality. Recall that for irreducible $\pi\in \Rep(G(F),\omega)$ with cuspidal support $(M,\sigma)$,  $d(\pi)$ (resp. $d^\prime(\pi)$) is the rank of the maximal split torus in $Z_M$ (resp. $Z_M\cap[G,G]$).
 \begin{thm}\cite[Theorem 1,2]{NP20}\label{SSNP} 
 Let $\pi\in\Rep(G(F),\omega)$ irreducible. Then for any $\pi^\prime\in\Rep(G(F))$, one has $\Ext^i_{G(F)}(\pi,\pi^\prime)=0$ for $i>d(\pi)$ and for $0\leq i\leq d(\pi)$, there is a non-degenerate pairing 
 $$\Ext^i_{G(F)}(\pi,\pi^\prime)\times \Ext^{d(\pi)-i}_{G(F)}(\pi^\prime, D(\pi))\to \Ext_{G(F)}^{d(\pi)}(\pi,D(\pi))\cong\BC.$$
If  $\pi^\prime\in \Rep(G(F),\omega)$, then $\Ext^i_{\Rep(G(F),\omega)}(\pi,\pi^\prime)=0$ for $i>d^\prime(\pi)$ and for $0\leq i\leq d^\prime(\pi)$, there is a non-degenerate pairing 
 $$\Ext^i_{\Rep(G(F),\omega)}(\pi,\pi^\prime)\times \Ext^{d^\prime(\pi)-i}_{\Rep(G(F),\omega)}(\pi^\prime, D(\pi))\to \Ext_{\Rep(G(F),\omega)}^{d^\prime(\pi)}(\pi,D(\pi))\cong\BC.$$
  \end{thm}
 \section{Relatively supercuspidal representation}
  Let $G$ be a reductive group over $F$ with center $Z_G$ and $H\subset G$ be a reductive
$F$-subgroup.  Take any smooth character $\omega:\ Z_GH(F)\to\BC^\times$. Let  $\eta:=\omega|_{H(F)}$ and denote $\omega|_{Z_G(F)}$ again by $\omega$. An irreducible representation $\pi\in\Rep(G(F),\omega)$ is called 
\begin{itemize}
	\item {\em ($(H,\eta)$-)distinguished} if
	\[\Hom_{H(F)}(\pi,\eta) = \Hom_{Z_GH(F)}(\pi,\omega)=
	\Hom_{G(F)}\left( \pi, I_{Z_GH(F)}^{G(F)}\omega \right) \not= 0.\]
	\item {\em ($(H,\eta)$-)relatively discrete series} (when $\omega$ is unitary) if $\pi$ is $(H,\eta)$-distinguished and 
	\[\Hom_{G(F)}\left( \pi, L^2(Z_GH(F) \bs G(F), \omega) \right)\neq0,\]
	where $L^2(Z_GH(F) \bs G(F), \omega)$  is the $L^2$-completion of $i_{Z_GH(F)}^{G(F)}\omega$.
	\item {\em ($(H,\eta)$-)relatively supercuspidal} if $\pi$ is distinguished and 
	\[\Hom_{G(F)}\left( \pi, I_{Z_GH(F)}^{G(F)}\omega \right)  = \Hom_{G(F)}
	\left( \pi, i_{Z_GH(F)}^{G(F)}\omega \right).\]
\end{itemize}
When $\eta$ is trivial, we usually omit it (or sometimes $(H,\eta)$) in notations. Note that if $\omega$ is unitary and  $\pi\in\Rep(G(F),\omega)$ is RSC, then $\pi$ is RDS.
\begin{lem}\label{td}For any irreducible $\pi\in\Rep(G(F),\omega)$ such that $\dim\Hom_{H(F)}(\pi,\BC)<\infty$, one has
 \[\dim\Hom_{H(F)}\left(\pi,\eta\right)\geq\dim\Hom_{G(F)}\left(\pi,i_{Z_GH(F)}^{G(F)}\omega\right)=\dim\Ext^{d^\prime(\pi)}_{\Rep(H(F),\eta^{-1}|_{Z_G\cap H(F)})}(D(\pi)^\vee,\eta^{-1}).\] 
If  $Z_G\cap H(F)\bs Z_G(F)$ is compact, one has 
\[\dim\Hom_{G(F)}\left(\pi,i_{Z_GH(F)}^{G(F)}\omega\right)=\dim\Ext^{d(\pi)}_{H(F)}(D(\pi)^\vee,\eta^{-1}).\] 
\end{lem}
 \begin{proof}Note that 
 $$\Hom_{H(F)}(\pi,\eta)=\Hom_{Z_GH(F)}(\pi,\omega)=\Hom_{G(F)}(\pi,I_{Z_GH(F)}^{G(F)}\omega)\supset\Hom_{G(F)}(\pi,i_{Z_GH(F)}^{G(F)}\omega).$$
 By the Schnieder-Stuhler duality in Theorem \ref{SSNP}, 
 \begin{align*}
 \dim\Hom_{G(F)}(\pi,i_{Z_GH(F)}^{G(F)}\omega)&=\dim\Ext^{d^\prime(\pi)}_{\Rep(G(F),\omega)}(i_{Z_GH(F)}^{G(F)}\omega,D(\pi))\\
 &=\dim\Ext_{\Rep(G(F),\omega^{-1})}^{d^\prime(\pi)}\left(D(\pi)^\vee, I_{Z_GH(F)}^{G(F)}\omega^{-1}\right).  
 \end{align*}
By Frobenius reciprocity and   the  category equivalence $$\Rep(Z_GH(F),\omega|_{Z_G(F)}^{-1})\cong \Rep(H(F), \omega^{-1}|_{Z_G\cap H(F)}),$$
one deduce
$$\Ext_{\Rep(G(F),\omega^{-1})}^{d^\prime(\pi)}\left(D(\pi)^\vee, I_{Z_GH(F)}^{G(F)}\omega^{-1}\right)=\Ext^{d^\prime(\pi)}_{\Rep(H(F),\eta^{-1}|_{Z_G\cap H(F)})}(D(\pi)^\vee,\eta^{-1})$$
When $Z_G\cap H(F)\bs Z_G(F)$ is compact, one has 
\[\Hom_{G(F)}\left(\pi,i_{Z_GH(F)}^{G(F)}\omega\right)=\Hom_{G(F)}\left(\pi, i_{H(F)}^{G(F)}\eta\right).\] 
By the Schnieder-Stuhler duality in Theorem \ref{SSNP} and Frobenius reciprocity agian,
\begin{align*}
  \dim\Hom_{G(F)}(\pi,i_{H(F)}^{G(F)}\omega)&=\dim\Ext^{d(\pi)}_{G(F)}(i_{H(F)}^{G(F)}\eta,D(\pi))\\
  &=\dim\Ext_{G(F)}^{d(\pi)}\left(D(\pi)^\vee, I_{H(F)}^{G(F)}\eta^{-1}\right)\\
  &=\dim\Ext^{d(\pi)}_{H(F)}\left(D(\pi)^\vee,\eta^{-1}\right).
\end{align*}
 \end{proof}
 \begin{prop}\label{tprsc}Let  $\pi\in\Rep(G(F),\omega)$ be an irreducible representation with $\dim\Hom_{H(F)}(\pi,\BC)<\infty$.  
 Then
 \[\dim \Hom_{H(F)}(\pi,\BC) \geq \dim\Ext^{d^\prime(\pi)}_{\Rep(H(F),\omega^{-1}|_{Z_G\cap H(F)})}(D(\pi)^\vee,\omega^{-1}),\]
and $\pi$ is RSC if and only if 
 \[0\neq\dim\Hom_{H(F)}\left( \pi, \eta \right)=\dim\Ext^{d^\prime(\pi)}_{\Rep(H(F),\omega^{-1}|_{Z_G\cap H(F)})}(D(\pi)^\vee,\omega^{-1}),\]
Moreover,
 \begin{itemize}
 \item if  $Z_G\cap H(F)\bs Z_G(F)$ is compact, $\pi$ is RSC if and only if
 \[0\neq\dim\Hom_{H(F)}\left( \pi, \eta \right)=\dim\Ext^{d(\pi)}_{H(F)}(D(\pi)^\vee,\omega^{-1}).\]
     \item if $\pi$ is  supercuspidal, then $\pi$ is RSC if and only if 
 \[0\neq\dim\Hom_{H(F)}\left( \pi, \eta \right)=\dim\Hom_{H(F)}\left(\pi^\vee,\eta^{-1}\right)\]
 \end{itemize}
 \end{prop}
 \begin{proof}The criterion for relatively supercuspidality follows directly from Lemma \ref{td}. When $\pi$ is supercuspidal, $d^\prime(\pi)=0$ and $D(\pi)=\pi$ by Proposition \ref{AuZe1}.  Thus the criterion follows.
 \end{proof}

\section{The Flicker-Rallis case} \label{sec:FR}
 In this section,   let $H=\GL_n\subset G=\Res_{E/F}\GL_n$. We start with the geometry of double cosets. Let $P=MN\subset G$ be any standard parabolic subgroup and  identify $W_M\bs W_G /W_M$ with the set ${}^MW_G^M$ of left and right $W_M$-reduced elements in $W_G$, where  $W_G$ (resp. $W_M$) is the Weyl group of $G$ (resp. $M$). Let $W_2=\{w\in W_G\mid w^2=1\}$ and  ${}^MW_2^M=\{w\in{}^MW_G^M\mid w^2=1\}$. As explained in \cite[Section 13]{JLR99}, every double coset in  $P(
F)\bs G(F)/H(F)$ is represented by $\eta\in G(F)$ such that $w=\eta\bar{\eta}^{-1}\in {}^MW^M_2$ and this gives a bijection 
$$P(F)\bs G(F)/H(F) \stackrel{\sim}{\lra} {}^MW^M_2.$$ 
Denote by
\begin{align*}
   P(w):=M\cap wPw^{-1},\quad M(w):=M\cap wMw^{-1},\quad N(w):=M\cap wNw^{-1}.
\end{align*}
Then by \cite[Section 2.4]{Off17}, $P(w)=M(w)N(w)$ is a standard parabolic subgroup of $M$. Moreover, let $\pr_M$ be the  projection from $P$ to $M$ and set
$$P_\eta:=\eta H\eta^{-1}\cap P;\quad M_\eta:=\eta H\eta^{-1}\cap M.$$
When $w=1$, we write $P_H$ (resp. $M_H$) instead of $P_\eta$ (resp. $M_\eta$). 

Let $N_\eta$ be the unipotent radical of $P_\eta$ and let $\delta_\eta$ (resp. $\delta_{P}$) be the modulus character of $P_\eta$ (resp. $P$). Then by \cite[Lemma 3.2, 3,3 \& Corollary 6.9]{Off17}, 
\begin{itemize}
    \item $P_\eta=M_\eta N_\eta$;
    \item $\pr_M(N_\eta)=N(w)$ is a normal subgroup of $\pr_M(P_\eta)=M_\eta N(w)\subset P(w)$;
    \item $\delta_{P}|_{P_\eta}=\delta_\eta^2$. 
\end{itemize} 
Note that $N(w)=\{e\}$ if and only if $w$ normalizes $M$, i.e. $M=w M w^{-1}$. In this case, the corresponding  double coset is called  \dfn{$P$-admissible} in \cite[Section 13]{JLR99} and one has  $N_\eta=\eta H\eta^{-1}\cap N$.
\begin{example}
Assume $n=3$ and $P=B$. Then there are $4$ double cosets, represented by $\eta_i$, $i=1,\ldots,4$
   with $w_i = w(\eta_i)$ as following
   \[w_1 = 1, \quad w_2 = \left(\begin{matrix} 0 & 1 & 0 \\ 1 & 0 & 0 \\ 0 & 0 & 1 \end{matrix}\right), \quad
     w_3 = \left(\begin{matrix} 1 & 0 & 0 \\ 0 & 0 & 1 \\ 0 & 1 & 0 \end{matrix}\right), \quad 
     w_4 = \left(\begin{matrix} 0 & 0 & 1 \\ 0 & 1 & 0 \\ 1 & 0 & 0 \end{matrix}\right).\]
   Note that
   \[B_\eta = \left\{ b \in B(F) \Big| w(\eta)^{-1} b w(\eta) = \bar{b} \right\}.\]
   The double $B(E)\eta_1\GL_3(F)$ is closed with $B_{\eta_1}=B\cap H$. 
   The double cosets $B(F)\eta_iH(F)$, $i=2,3$  are locally closed with
   \[M_{\eta_2} = \left\{\left(\begin{matrix} \alpha & 0 & 0 \\ 0 & \bar{\alpha} & 0 \\ 0 & 0 & \beta 
	   \end{matrix}\right) \Big| \alpha \in E^\times, \beta \in F^\times \right\}, \quad
     N_{\eta_2} = \left\{\left(\begin{matrix} 1 & 0 & u \\ 0 & 1 &\bar{u}  \\ 0 & 0 & 1 
	   \end{matrix}\right) \Big| u \in E^\times\right\}\]
   \[M_{\eta_3} = \left\{\left(\begin{matrix} \beta & 0 & 0 \\ 0 & \alpha & 0 \\ 0 & 0 & \bar{\alpha} 
	   \end{matrix}\right) \Big| \alpha \in E^\times, \beta \in F^\times \right\}, \quad
	   N_{\eta_3} = \left\{\left(\begin{matrix} 1 & u & \bar{u} \\ 0 & 1 & 0   \\ 0 & 0 & 1 
	   \end{matrix}\right) \Big| u \in E^\times\right\}.\]
  The double coset $B(F)\eta_4H(F)$ is open with
  \[M_{\eta_4} = \left\{\left(\begin{matrix} \alpha & 0 & 0 \\ 0 & \beta & 0 \\ 0 & 0 & \bar{\alpha} 
  \end{matrix}\right) \Big| \alpha \in E^\times, \beta \in F^\times \right\}, \quad N_{\eta_4} = 1.\]
\end{example}
 Let  $\pi=I_{P(F)}^{G(F)}\sigma$ for some irreducible $\sigma\in \Rep(M(F))$ with  $M=\prod_{i=1}^d\Res_{E/F}\GL_{n_i}$. 
 As mentioned in  the introduction, one has the following exact sequence of $H(F)$-representations 
\[\tag{E} 0\lra \tau\lra \pi\lra I_{P_H(F)}^{H(F)}\delta_{P_H}^{1/2}\sigma|_{M_H(F)}\lra 0\]
where $\tau$ has a filtration with graded pieces $$V_\eta:=i_{H(F)\cap \eta^{-1}P(F)\eta}^{H(F)}(\delta_{P_\eta}^{1/2}\sigma|_{P_\eta(F)})^\eta$$
with $w=\eta\bar{\eta}^{-1}\neq1\in {}^MW_2^M$. Here
$(-)^\eta$ is the conjugation functor that transfers $P_\eta(F)$-representations into $\eta^{-1}P(F)\eta\cap H(F)$-representations. 
\begin{lem}\label{com1}One has $\dim\Ext^{d(\pi)}_{H(F)}(I_{P_H(F)}^{H(F)}\delta_{P_H}^{1/2}\sigma|_{M_H(F)},\BC)\leq 1$ with the equality holds if and only if  $D(\sigma)^\vee$ is $M\cap H$- RSC.  
\end{lem} 
\begin{proof}By Frobenius reciprocity law, 
$$\Ext^{d(\pi)}_{H(F)}(I_{P_H(F)}^{H(F)}\delta_{P_H}^{1/2}\sigma|_{M_H(F)},\BC)\cong\Ext^{d(\pi)}_{ M_H(F)}(\delta_{P_H(F)}^{1/2}\sigma|_{M_H(F)},\delta^{1/2}_{P_H(F)})\cong\Ext^{d(\pi)}_{M_H(F)}(\sigma|_{M_H(F)},\BC)$$
 Then  by  the Kunneth formula in Proposition \ref{Kun} and the fact $d(\pi)=\sum_{i=1}^d d(\sigma_i)$,
$$\Ext_{M_H(F)}^{d(\pi)}(\sigma|_{M_H(F)},\BC)=\otimes_{i=1}^d\Ext^{d(\sigma_i)}_{\GL_{n_i}(F)}(\sigma_i,\BC).$$
Denote the central character of $\sigma_i$ by $\omega_i$. Note that by Theorem \ref{SSNP}, 
$$\dim_\BC\Ext^{d(\sigma_i)}_{\GL_{n_i}(F)}(\sigma_i,\BC)=\dim_\BC\Hom_{\GL_{n_i}(F)}(i_{\GL_{n_i}(F)}^{\GL_{n_i}(E)}\BC,D(\sigma_i))$$
which is clearly zero unless $\omega_i|_{F^\times}=1$. When $\omega_i|_{F^\times}=1$,  by Lemma \ref{td} 
$$\dim\Ext^{d(\sigma_i)}_{\GL_{n_i}(F)}(\sigma_i,\BC)=\dim\Hom_{\GL_{n_i}(E)}(D(\sigma_i)^\vee, i_{E^\times \GL_{n_i}(F)(F)}^{\GL_{n_i}(E)}\omega_i^{-1})$$
which is $\leq 1$ with the equality holds if and only if $D(\sigma_i)^\vee$ is RSC for the pair $(\Res_{E/F}\GL_{n_i},\GL_{n_i})$.
\end{proof}
\begin{lem}\label{com2}For any $\eta$ such that $w=\eta\bar{\eta}^{-1}\neq1$, $\Ext^{d(\pi)}_{H(F)}(V_\eta,\BC)=0$ for $\sigma$  supercuspidal and  $\Ext^{d(\pi)-1}_{H(F)}(V_\eta,\BC)=0$ for  $\sigma$  regular supercuspidal and $H\cap M$-distinguished.
\end{lem} 
\begin{proof}By duality and the Frobenius reciprocity law, for each $i\in\BN$ $$\Ext^{i}_{H(F)}(V_\eta,\BC)\cong\Ext^{i}_{\eta H(F)\eta^{-1}}(i_{P_\eta(F)}^{\eta H(F)\eta^{-1}}(\delta_{P_\eta}^{1/2}\sigma|_{P_\eta(F)}),\BC)\cong\Ext^i_{ P_\eta(F)}(\sigma|_{P_\eta(F)},\BC)$$  By the adjunction between parabolic induction and Jacquet functor, one has $$\Ext^i_{ P_\eta(F)}(\sigma|_{P_\eta(F)},\BC)
\cong\Ext^i_{M_\eta(F)}((\sigma|_{P_\eta(F)})_{N_\eta(F)},\BC)\cong\Ext^i_{M_\eta(F)}((\sigma|_{P(w)(F)})_{N(w)(F)},\BC)$$
where $(-)_*$ means taking co-invariant of $*$. 

If $N(w)\neq\{e\}$,  $(\sigma|_{M_\eta(F)N(w)(F)})_{N(w)(F)}$=0  and consequently $\Ext^{i}_{H(F)}(V_\eta,\BC)=0$ for any $i\in\BN$ since $\sigma$ is supercuspidal. 

When $N(w)=\{e\}$, one has  $wMw^{-1}=M$. As explained in \cite[Section 2.1]{G}, there is an involution $\epsilon\in S_d$ such that $n_i=n_{\epsilon(i)}$ and $$M_\eta(F)=\{(g_i)\in M(F)\mid g_{\epsilon(j)}=\bar{g_j}\}.$$ 
 For any $1\leq i\leq d$,
 \begin{itemize}
     \item if $\epsilon(j)=j$,  let $\tilde{\sigma}_j:=\sigma_j$, $H_j:=\GL_{n_j}(F)$ and $G_j:= \GL_{n_j}(E)$;
     \item if $\epsilon(j)\neq j$, let $\tilde{\sigma}_j:=\sigma_j\boxtimes\sigma_{\epsilon(j)}$, $H_j:= \{(g,\bar{g})\mid g\in\GL_{n_j}(E)\}$ and $G_j:= \GL_{n_j}(E)\times\GL_{n_{\epsilon(j)}}(E)$
 \end{itemize} 
Let $J\subset \{1,\cdots,d\}$ represent the  orbits of $\epsilon$. The assumption $\eta\bar{\eta}^{-1}\neq1$ implies $\epsilon\neq1$ and $|J|<d=d(\pi)$. Then by the Kunneth formula in Proposition \ref{Kun}, 
 \begin{align*}
\Ext^{i}_{M_\eta(F)}(\sigma, \BC)\cong\oplus_{j\in J, \sum n(j)=i}\Ext^{n(j)}_{H_j}(\tilde{\sigma}_j,\BC)
 \end{align*}
  Note that when  $\epsilon(j)\neq j$, one has
    $$\Ext^i_{H_j}(\tilde{\sigma}_j, \BC)\cong\Ext^i_{\GL_{n_j}(E)}(\sigma_j\otimes\bar{\sigma}_{\epsilon(j)},\BC)\cong\Ext^i_{\GL_{n_j}(E)}(\sigma_j,\bar{\sigma}_{\epsilon(j)}^\vee)$$
Thus
 $\Ext^i_{H_j}(\tilde{\sigma}_j, \BC)=0$ for $i> 1$ by Theorem \ref{SSNP} and  consequently, $\Ext^{d(\pi)}_{M_\eta(F)}(\sigma, \BC)=0$.

When  $\sigma$ is $H\cap M$-distinguished,  $\sigma_i\cong\bar{\sigma}^\vee_i$ by \cite{Fli91}. If $\pi$ is moreover supercuspidal and  regular  for $G$,  then 
 $$\dim_\BC\Ext^1_{\GL_{n_j}(E)}(\sigma_j,\bar{\sigma}_{\epsilon(j)}^\vee)=\dim_\BC\Hom_{\GL_{n_j}(E)}(\sigma_{\epsilon(j)},\sigma_j)=0.$$
 for $\epsilon(j)\neq j$. Thus $\Ext^{d(\pi)-1}_{M_\eta(F)}(\sigma, \BC)=0$.
\end{proof} 
\begin{prop}\label{reRSC}If $\pi = I_{P(F)}^{G(F)} \sigma$ where 
$\sigma\in\Rep(M(F))$ is  $H\cap M$-distinguished regular supercuspidal, then  $$\dim\Ext^{d(\pi)}_{H(F)}(\pi,\BC)=\dim\Ext^{d(\pi)}_{H(F)}(\pi^\vee,\BC)=1.$$
Consequently,  both $\pi$ and $\pi^\vee$ are RSC.
\end{prop}
\begin{proof}By Theorem \ref{BZ} and Proposition \ref{AuZe2}, $\pi$ is irreducible and $D(\pi)=\pi$. Note that $\pi\cong\bar{\pi}^\vee$ by \cite{Fli91}, so   the same holds for $\pi^\vee=\pi=I_{P(F)}^{G(F)}\sigma^\vee$. Note that $\dim \Hom_{H(F)}(\pi,\BC) \leq 1$. 
By Proposition \ref{tprsc}, it suffices to show 
$$\dim\Ext^{d(\pi)}_{H(F)}(\pi,\BC)=\dim\Ext^{d(\pi)}_{H(F)}(\pi^\vee,\BC)=1.$$

From the short exact sequence (E) above, one deduce a long exact sequence
$$\cdots\to\Ext^{d(\pi)-1}_{H(F)}(\tau,\BC)\to \Ext^{d(\pi)}_{H(F)}(I_{P_H(F)}^{H(F)}\delta_{P_H}^{1/2}\sigma|_{M_H(F)},\BC)\to\Ext^{d(\pi)}_{H(F)}(\pi,\BC)\to 
\Ext^{d(\pi)}_{H(F)}(\tau,\BC)\to0$$
By Lemma \ref{com1},  $\Ext^{d(\pi)}_{H(F)}(I_{P_H(F)}^{H(F)}\delta_{P_H}^{1/2}\sigma|_{M(F)},\BC)=\BC$. By Lemma \ref{com2},  $$\Ext^{d(\pi)-1}_{H(F)}(\tau,\BC)=\Ext^{d(\pi)}_{H(F)}(\tau,\BC)=0.$$  Hence $\Ext^{d(\pi)}_{H(F)}(\pi,\BC)=\BC$. Similarly, $\Ext^{d(\pi)}_{H(F)}(\pi^\vee,\BC)=\BC$.
\end{proof}
\begin{prop}\label{AZRSC}If $\pi$ is RSC, then $D(\pi)=\pi$.
\end{prop}
\begin{proof}Take any injection $D(\pi)^\vee\hookrightarrow I_{P(F)}^{G(F)}\sigma$ with $P=M N\subset G$ standard and $\sigma\in\Rep(M(F))$ cuspidal.
 By Proposition \ref{tprsc},  $\Ext^{d(\pi)}_{H(F)}(I_{P(F)}^{G(F)}\sigma,\BC)\neq0$. By the above short exact sequence and Lemma  \ref{com2}, $\Ext^{d(\pi)}_{H(F)}(I_{P_H(F)}^{G(F)}\delta_{P_H}^{1/2}\sigma|_{M_H(F)},\BC)\neq0.$
 By Lemma \ref{com1}, $\sigma$ is distinguished. By Theorem \ref{BZ}, $I_{P(F)}^{G(F)}\sigma$ is  irreducible and hence $D(\pi)=I_{P(F)}^{G(F)}\sigma^\vee$. By Proposition \ref{AuZe2}, $\pi=D(\pi)$.
\end{proof}

\begin{prop}\label{descent}If $\pi=I_{P(F)}^{G(F)}\sigma$  with $\sigma\in\Rep(M(F))$ square-integrable is  RSC, then $\sigma$ must be distinguished and supercuspidal. In particular, RSC discrete series are supercuspidal.
\end{prop}
\begin{proof}By assumption, $\pi=D(\pi)=I_{P^{-}(F)}^{G(F)}D(\sigma) = I_{P(F)}^{G(F)} D(\sigma)$. 
Here, the last equation is from  Theorem \ref{BZ}. Being induced from a discrete series, $\pi$ is generic. Hence
$D(\sigma)$ must also be generic, which is the case only if $\sigma$ itself is cuspidal.
\end{proof}

We now can prove Theorem \ref{topd}. For the convenience of readers, we restate the theorem here.
\begin{thm}  Let $\pi\in \Rep(G(F))$ be an
 irreducible representation. Then $\pi$   is RSC if and only if 
$\pi = I_{P(F)}^{G(F)}\sigma$ 
for some parabolic subgroup $P=MN \subset G$ and some  
$H \cap M$-distinguished regular supercuspidal $M(F)$-representation $\sigma$.  
\end{thm}
\begin{proof}
The if part is established in Proposition \ref{reRSC}. For the only if part, let $\pi\in\Rep(G(F))$ be RSC, in particular RDS. By the Plancherel decomposition
of $L^2(H \bs G)$ established by Beuzart-Plessis in \cite{BP18P} and by the work of
Mok \cite{Mok} (see the brief review in the introduction),  , $\pi=I_{P(F)}^{G(F)} \sigma$ for some regular distinguished  square-integrable $\sigma\in\Rep(M(F))$. 
By Proposition \ref{descent}, $\sigma$ is actually supercuspidal and we are done.
\end{proof}
To conclude this section, we record another approach, suggested by Prof. D. Prasad, to Corollary \ref{Stein-sub} in the case $n=2$ by computing Jacquet modules. 
	\begin{prop}\label{prasad}Let $G=\Res_{E/F}\GL_2$ and $H=\GL_2$. 
	For $\pi=I_{B(F)}^{G(F)}\chi_1\boxtimes\chi_2\in\Rep(G(F))$ irreducible,  $$\dim\Hom_{H(F)}(\St,\pi)\leq1$$
	with the equality holds if and only if $\chi_1\neq\chi_2$ and $\chi_1|_{F^\times}=\chi_2|_{F^\times}=1$.
			\end{prop}
			\begin{proof}By the Mackey theory, there exists an exact sequence of $H(F)$-representations
	$$0\to \BC\to I_{B_H(F)}^{H(F)}\delta_{B_H}^{-1/2}\to\St\to0.$$
which leads to a long exact sequence
	$$0\to\Hom_{H(F)}(\St,\pi)\to\Hom_{H(F)}(I_{B_H(F)}^{H(F)}\delta_{B_H}^{-1/2},\pi)\to \Hom_{H(F)}(\BC,\pi)\to\cdots$$
Let $\CK(\pi)$ be the Kirillov model of $\pi$ with respect to any nontivial additive character $\psi:\ F\bs E\to\BC^\times$. Then clearly $\Hom_{H(F)}(\BC,\CK(\pi))=0$, and hence $$\Hom_{H(F)}(\St,\pi)=\Hom_{H(F)}(I_{B_H(F)}^{H(F)}\delta_{B_H}^{-1/2},\pi).$$

Let $B=TN$ be standard decomposition. 
Then  by Bernstein's second adjointness theorem, 
$$\Hom_{H(F)}(I_{B_H(F)}^{H(F)}\delta_{B_H}^{-1/2},\pi)=\Hom_{T_H(F)}(\delta_{B_H}^{-1}, \pi_{N_H^-(F)})=\Hom_{T_H(F)}(\delta_{B_H}, \pi_{N_H(F)}).$$
Note that one has an exact sequence of $B(F)$-representations
$$0\to\CC_c^\infty(E^\times)\to\CK(\pi)\to \pi_{N(F)}\to0$$
Let $\CK_F$	be the image of $\CK(\pi)$ under the restriction (to $F^\times$) map and $J_F$ be the cokernel of $\CC_c^\infty(F^\times)\hookrightarrow \CK_F$. Then by taking $N_H(F)$-coinvariants, one gets a commutative diagram of $T_H(F)$-modules
$$\xymatrix{ 0 \ar[r] & \CC_c^\infty(F^\times)   \ar[d]^{=} \ar[r] & \CK(\pi)_{N_H(F)} \ar[d]^{\res} \ar[r] & \pi_{N(F)}  \ar[d] \ar[r]& 0\\
 0 \ar[r] & \CC_c^\infty(F^\times)   \ar[r] &\CK_F  \ar[r] & J_F \ar[r] & 0 }$$
By \cite[Theorem 4.7.2]{Bum97},  
\[ 
\dim J_F=\begin{cases} 1 \ & \text{if}\ \chi_1|_{F^\times}= \chi_2|_{F^\times}\ \text{and}\ \chi_1\neq\chi_2;\\ 
	2 \ & \text{otherwise}\end{cases}\]
By \cite[Lemma 6.3]{PT11}, $\CC_c^\infty(E^\times)_{N_H(F)}=\CC_c^\infty(F^\times)$. Since $\res$ is surjective by definition, we deduce by the  Snake Lemma that 
\begin{itemize}
	\item if  $\chi_1=\chi_2$ or $\chi_1|_{F^\times}\neq \chi_2|_{F^\times}$, $\CK(\pi)_{N_H(F)}=\CK_F$;
	\item if $\chi_1|_{F^\times}=\chi_2|_{F^\times}$ and $\chi_1\neq\chi_2$, $\ker(\res)\cong \chi|_{T_H(F)}\delta_{B_H}$ is one-dimensional
\end{itemize}
Since $\CK_F$ contains no one-dimensional $T_H(F)$-subrepresentations, we deduce that 
 $\dim \Hom_{H(F)}(\St,\pi)\leq 1$ and the equality holds if and only if  $\chi_1|_{F^\times}=\chi_2|_{F^\times}=1$ and $\chi_1\neq\chi_2$
 from the exact sequence \[0\to \ker (\res)\to \CK(\pi)_{N_H(F)}\to \CK_F\to0.\]

			\end{proof}
\section{The diagonal case}\label{sec:dia}
 Let $G_2\subset G_1$ are reductive groups over $F$.  In this section, we will consider the diagonal case $G=G_1\times G_2$ and view $H=G_2$ as a subgroup of $G$ via the diagonal embedding. 
\begin{prop}\label{RSCsc}If $\pi=\pi_1\boxtimes \pi_2\in\Rep(G(F),\omega)$ is RSC, then $\pi_2$ is supercuspidal and $d^\prime(\pi_1)\leq d(\pi_2)$. If moreover $(Z_G\cap H(F))\bs Z_H(F)$ is compact, then $\pi$ is  supercuspidal.
 \end{prop}
 \begin{proof}
 Assume  $\pi_i\in\Rep(G_i(F),\omega_i)$.   Then by  Proposition \ref{Ext},
 \begin{align*}
     \Ext^{d^\prime(\pi)}_{\Rep(H(F),\eta^{-1}|_{Z_G\cap H(F)})}(D(\pi)^\vee,\eta^{-1})&=\Ext^{d^\prime(\pi)}_{\Rep(G_2(F), \omega_2\eta^{-1}|_{Z_{G_1}\cap G_2(F)})}(D(\pi_1)^\vee, D(\pi_2)\otimes\eta^{-1})\\
 &=\Ext^{d^\prime(\pi)}_{\Rep(Z_{G_1}(F)G_2(F),\omega_1^{-1})}(D(\pi_1)^\vee, D(\pi_2)\otimes\eta^{-1})\\
 &=\Ext^{d^\prime(\pi)}_{\Rep(G_1(F),\omega^{-1}_1)}(D(\pi_1)^\vee, I_{Z_{G_1}(F)G_2(F)}^{G_1(F)}D(\pi_2)\otimes\eta^{-1})
 \end{align*}
where in the last row, $Z_{G_1}(F)$ acts on $D(\pi_2)\otimes\eta^{-1}$ by $\omega_i^{-1}$. 

Note that $d^\prime(\pi)=d^\prime(\pi_1) +d^\prime(\pi_2)$ and $$\Ext^{d^\prime(\pi)}_{\Rep(G_2(F), \omega_2\eta^{-1}|_{Z_{G_1}\cap G_2(F)})}(D(\pi_1)^{\vee}, D(\pi_2)\otimes\eta^{-1})\subset \Ext^{d^\prime(\pi)}_{G_2(F)}(D(\pi_1)^\vee, D(\pi_2)\otimes\eta^{-1})$$
By Proposition \ref{tprsc} and Theorem \ref{SSNP}, $\pi$ is RSC implies $d^\prime(\pi_2)=0$ and $d^\prime(\pi_1)\leq d(\pi_2)$.

When $Z_G\cap H(F)\bs Z_H(F)$ is compact,
$$\Ext^{d^\prime(\pi)}_{\Rep(G_2(F), \omega_2\eta^{-1}|_{Z_{G_2}(F)})}(D(\pi_1)^{\vee,\prime}, D(\pi_2)\otimes\eta^{-1})\cong \Ext^{d^\prime(\pi)}_{\Rep(G_2(F), \omega_2\eta^{-1}|_{Z_{G_1}\cap G_2(F)})}(D(\pi_1)^\vee, D(\pi_2)\otimes\eta^{-1})$$
where $D(\pi_1)^{\vee,\prime}\subset D(\pi_1)^{\vee}$ is the direct summand on which $Z_{G_2}(F)$ acts via $\omega_2\eta^{-1}$.  Consequently, $d^\prime(\pi)\leq d^\prime(\pi_2)=0$ and $\pi$ is supercuspidal.
\end{proof}
 \begin{thm}\label{diagonal2} Let $G=G_1\times G_2$ and $H= G_2$ viewed as a subgroup of $G$ via the diagonal embedding. 
 Assume $Z_G\cap H(F)\bs Z_H(F)$ is compact and 
 \begin{itemize}
     \item either $(G,H)$ is symmetric;
     \item or the geometric quotient $X:=H\bs G$ is a wavefront spherical variety and $G$ is split.
 \end{itemize} 
 For any  irreducible $\pi\in\Rep(G(F))$, $\pi$ is RSC  if and only if $\pi$ is distinguished and supercuspidal.
 \end{thm}
 \begin{proof}The if part is just Proposition  \ref{RSCsc}. The only if part follows from \cite[lemma 3.1]{Zha0}.
 \end{proof}
Finally we consider the case $G_1=\GL_{n+1}$ and $G_2=\GL_n$ where $G_2$ is viewed as a subgroup of $G_1$ via the embedding $$\GL_n\hookrightarrow\GL_{n+1},\quad g\mapsto \begin{pmatrix} g & 0 \\ 0 & 1\end{pmatrix}.$$
It is well-known the pair $(G,H)$ is Gelfand and generic $\pi\in\Rep(G(F))$ is distinguished. 
\begin{thm}\label{GGP2} Let  $\pi=\pi_1\boxtimes\pi_2\in\Rep(G(F))$ be  an irreducible representation. Then $\pi$ is  RSC if and only if\begin{itemize}
     \item $\pi$ is supercuspidal when $n\geq2$,
     \item $\pi$ is  supercuspidal or $\pi_1=\St\otimes\eta^{-1}$ and $\pi_2=\eta$ for some character $\eta$ on $F^\times$ when $n=1$.
 \end{itemize}
 \end{thm}
 \begin{proof}
Assume $\pi$ is RSC. By Proposition \ref{RSCsc},  $\Ext^1_{H(F)}(D(\pi_1^\vee), D(\pi_2)) \neq 0$, $\pi_2$ is supercuspidal and $d^\prime(\pi_1)\leq 1$. Then by \cite[Theorem 4.2]{Pra18}, $$\EP_H(D(\pi)^\vee):=\dim\Hom_{H(F)}(D(\pi_1)^\vee,D(\pi_2))-\dim\Ext^1_{H(F)}(D(\pi_1)^\vee,D(\pi_2))
\leq 1$$
with the equality holds if and only if $D(\pi_1^\vee)$ is generic.  When $D(\pi_1^\vee)$ is generic, it is well-known  
 $$\dim\Hom_{H(F)}(D(\pi_1)^\vee,D(\pi_2))=\dim\Hom_{H(F)}(D(\pi^\vee),\BC)=1.$$
Assume $d^\prime(\pi_1)=1$. Then $\Ext^1_{H(F)}(D(\pi_1)^\vee,D(\pi_2))\neq0$ and hence $D(\pi_1^\vee)$ is non-generic. 
Note that $d(\pi_1)=d^\prime(\pi_1)+1=2$, so $\pi_1^\vee\subset I_{P(F)}^{\GL_{n+1}(F)}\sigma$ for some segment $\sigma$ and block 
$(\frac{n+1}{2},\frac{n+1}{2})$-parabolic $P$ by Theorem \ref{BZ} and Proposition \ref{AuZe2}. On the other hand, by Theorem \ref{SSNP}
  $$\dim\Ext^1_{H(F)}(D(\pi_1^\vee), D(\pi_2))=\dim\Hom_{H(F)}(\pi_2, D(\pi_1^\vee))$$
  Thus  $\pi_2$ must  appear in un-normalized Jacquet module $D(\pi^\vee)_{N(F)}$  with respect to the $(n,1)$-parabolic subgroup of $\GL_{n+1}$ by \cite[Proposition 5.3]{Pra18}. This forces $n=1$ and   $\pi_1=\St\otimes\eta^{-1}$ and $\pi_2=\eta$ for some character $\eta:\ F^\times\to \BC^\times$.  

When $\pi$ is supercuspidal,  $\pi$ is RSC by \cite[lemma 3.1]{Zha0}. 
When $\pi_1=\St\otimes\eta^{-1}$ and $\pi_2=\eta$, one has  $d^\prime(\pi)=1$ and $\Ext^{d^\prime(\pi)}_{H(F)}(D(\pi_1^\vee), D(\pi_2))=\BC$. Hence $\pi$ is RSC and we are done.
 \end{proof}

	\end{document}